\documentclass[reqno]{amsart}
\usepackage{amsbsy}
\usepackage{amsfonts}
\usepackage{amsmath}
\usepackage{amsthm}
\usepackage{amssymb}
\usepackage{enumitem}
\usepackage{comment}
\usepackage{pgfplots}
\usepackage{graphicx}

\usepackage{hyperref}
\usepackage{color}
\usepackage{etoolbox}
\expandafter\patchcmd\csname\string\proof\endcsname
  {\normalparindent}{0pt }{}{}
\allowdisplaybreaks

\theoremstyle{plain}

\newtheorem{dfn}{Definition}
\newtheorem{lem}{Lemma}
\newtheorem{prop}{Proposition}
\newtheorem{thm}{Theorem}

\newtheorem{conj}{Conjecture}

\newcommand{\Res}{\mathop{\mathrm{Res}}}

\newcommand{\E}{\mathbb{E}}
\newcommand{\Var}{\mathrm{Var}}
\newcommand{\vol}{\mathrm{vol}}
\newcommand{\q}{\mathfrak{q}}
\newcommand{\p}{\mathfrak{p}}
\newcommand{\bee}{\mathfrak{b}}

\newcommand{\Addresses}{{
  \footnotesize
  \bigskip
  \footnotesize

  \textsc{Department of Mathematics, Stanford University, Palo Alto, CA 94305, USA}\par\nopagebreak
  \textit{E-mail address:} \texttt{viviank@stanford.edu}

  \medskip

  \textsc{Department of Mathematics and Statistics, Queen's University, Kingston, Ontario, K7L 3N6, Canada}\par\nopagebreak
  \textit{E-mail address:} \texttt{brad.rodgers@queensu.ca}

  \medskip

  \textsc{Department of Applied Mathematics, H.I.T. - Holon Institute of Technology, Holon 5810201, Israel}\par\nopagebreak
  \textit{E-mail address:} \texttt{rodittye@gmail.com}

}}

\begin{document}

\title[Short intervals in algebraic number fields]{Sums of singular series and primes in short intervals in algebraic number fields}
\author{Vivian Kuperberg, Brad Rodgers, and Edva Roditty-Gershon}
\address{}
\email{}
\date{}

\begin{abstract}
Gross and Smith have put forward generalizations of  Hardy - Littlewood twin prime conjectures for algebraic number fields. We estimate the behavior of sums of a singular series that arises in these conjectures, up to lower order terms. More exactly, where $\mathfrak{S}(\eta)$ is the singular series, we find asymptotic formulas for smoothed sums of $\mathfrak{S}(\eta)-1$.

Based upon Gross and Smith's conjectures, we use our result to suggest that for large enough `short intervals' in an algebraic number field $K$, the variance of counts of prime elements in a random short interval deviates from a Cram\'{e}r model prediction by a universal factor, independent of $K$. The conjecture over number fields generalizes a classical conjecture of Goldston and Montgomery over the integers. Numerical data is provided supporting the conjecture.
\end{abstract}


\maketitle

\section{Introduction}
\label{Introduction}

\subsection{Background and motivation: the integers}
\label{BackgroundIntegers}
A motivation for this paper is a classical conjecture made by Goldston and Montgomery \cite{GoMo} regarding counts of primes in short intervals. One purpose of this note will be to put forward an analogue of this conjecture in other algebraic number fields and support this conjecture with some numerical data. Perhaps surprisingly, when put in the right form, these conjectures are essentially the same regardless of what number field one works in. We recall the conjecture for the integers:

\begin{conj}[Goldston-Montgomery]
\label{conj:GoldstonMontgomery}
For $\delta \in (0,1)$, and $H = X^\delta$,
\begin{equation}
\label{GoldstonMontgomery}
\frac{1}{X} \int_0^X \left( \sum_{x < n \leq x+H}\Lambda(n) - H \right)^2\, dx \sim H (\log X - \log H),
\end{equation}
as $X\rightarrow\infty$.
\end{conj}

Here $\Lambda(n)$ is the von Mangoldt function.
The intervals $\{n\in \mathbb{Z} \,: x < n \leq x+H\}$ are referred to as \emph{short intervals}; we say they are \emph{short} because their length $H$ is smaller than the numbers $x$ around which they occur.

Note that the left hand side of \eqref{GoldstonMontgomery} is approximately the probabilistic variance of the sums 
\begin{equation}
\label{shortintervalLambda}
\sum_{x < n \leq x+H} \Lambda(n)
\end{equation}
where $x$ is a random variable chosen at uniform from the interval $[0,X]$. Indeed, for $x$ chosen randomly in this way, the expected value of \eqref{shortintervalLambda} is approximately $H$.

For what follows, it will be natural to reformulate Conjecture \ref{conj:GoldstonMontgomery} into a claim about direct counts of primes in short intervals, rather than sums of the von Mangoldt function. We show in Appendix \ref{TwoWeights} that Conjecture \ref{conj:GoldstonMontgomery} is equivalent to Conjecture \ref{conj:primevariance} to be stated momentarily. We make use of the following set up:
\begin{itemize}
\item Define 
$$\pi(x;H) := \# \{p\; \textrm{prime}:\, x < p \leq x+H\}$$
to be the number of primes in the short interval of size $H$, around $x$. 
\item Define the expectation of this quantity by
$$
E_\mathbb{N}(X;H) := \frac{1}{X} \int_0^{X} \pi(x,H)\, dx.
$$
\item Note that by the prime number theorem $\sum_{\substack{x < n \leq x+H \\n\geq  2}} 1/\log n$ is a good approximation for $\pi(x;H)$. Consider
$$
\widetilde{\pi}(x;H):= \pi(x;H)-\sum_{\substack{x < n \leq x+H \\ n\geq  2}} \frac{1}{\log n}
$$
which keeps track of the deviation of $\pi(x;H)$ from this approximation.
\item Define
$$
V_{\mathbb{N}}(X;H):= \frac{1}{X} \int_0^{X} \widetilde{\pi}(x;H)^2 \, dx.
$$
This quantity differs slightly from the probabilistic variance of the count $\pi(x;H)$ as $x$ ranges randomly from $0$ to $X$, but it should be thought of as capturing the same information.
\end{itemize}
It follows from the prime number theorem that uniformly for $0 \leq H \leq X$,
$$
E_\mathbb{N}(X;H) \sim \frac{H}{\log X}
$$
as $X \rightarrow \infty$. For the variance we have:
\begin{conj}[The Goldston-Montgomery conjecture reformulated]
\label{conj:primevariance}
For $\delta \in (0,1)$ and $H = X^\delta$, we have
\begin{equation}
\label{primevarianceZ}
V_{\mathbb{N}}(X;H) \sim (1-\delta) \, E_\mathbb{N}(X;H).
\end{equation}
\end{conj}

The conjecture of Goldston and Montgomery has attracted interest for at least two reasons.  One: Goldston and Montgomery (building upon work of Gallagher and Mueller \cite{GaMu}) showed that it is connected to the local spacing of zeros of the Riemann zeta function. Two: the prediction \eqref{primevarianceZ} deviates rather starkly from the prediction one would make based off the Cram\'{e}r random model for primes. Under the Cram\'{e}r model one expects\footnote{Indeed, the Cram\'{e}r model suggests that the distribution of $\pi(x;H)$ for $x \in [0,X]$ chosen randomly, should be approximated by the distribution of a sum of random variables $X_1 + X_2 + \cdots + X_H$, where each $X_i$ is an i.i.d copy of a Bernoulli random variable, taking $1$ with probability $1/\log X$. Note that $\E (X_1 + \cdots + X_H) = \Var (X_1 + \cdots + X_H) = H/\log X$. Background information about the Cram\'{e}r model and its relation to this problem can be found in \cite{So}.} that 
\begin{equation}
\label{CramerModelIntegers}
V_\mathbb{N}(x;H) \sim E_{\mathbb{N}}(X;H).
\end{equation}

Indeed, it is common to justify a belief in this conjecture of Goldston and Montgomery by making use of the following conjecture of Hardy and Littlewood \cite{HaLi} (which itself also deviates from the Cram\'er model).
\begin{conj}[Hardy-Littlewood]
\label{conj:HardyLittlewoodIntegers}
For fixed integers $h \neq 0$,
\begin{equation}
\label{HardyLittlewoodIntegers}
\sum_{n\leq X} \mathbf{1}_\mathcal{P}(n)\mathbf{1}_\mathcal{P}(n+h) \sim \mathfrak{S}(h) \sum_{2 \leq n \leq X} \frac{1}{\log^2 n}
\end{equation}
where
$$
\mathfrak{S}(h):= \prod_p \frac{(1- \nu_h(p)/p)}{(1-1/p)^2}, \quad \textrm{with} \quad \nu_h(p) = \begin{cases} 2 & \textrm{if}\; h \not\equiv 0 \,( \mathrm{mod} \ p) \\ 1 & \textrm{if} \; h \equiv 0 \,(\mathrm{mod} \ p) \end{cases}
$$
\end{conj}
Here $\mathbf{1}_\mathcal{P}(n)$ is the function that is $1$ when $n$ is prime and $0$ otherwise, and the product is over all primes $p$. We note that the Cram\'{e}r model would make the almost certainly false (but still not disproven) prediction that the right hand side of \eqref{HardyLittlewoodIntegers} is just $\sum_{2\leq n \leq X} \frac{1}{\log^2 n}$. Of course the right hand side of \eqref{HardyLittlewoodIntegers} can be rewritten $\sim \mathfrak{S}(h)x/\log^2 X$, but the expression we have written is expected to be a more accurate approximation.

The heuristic derivation of Conjecture \ref{conj:GoldstonMontgomery} based upon Conjecture \ref{conj:HardyLittlewoodIntegers} can be found in \cite{MoSo1}. A key point of that derivation, which we record because we will prove an analogue of it, is the following estimate for sums of the function $\mathfrak{S}(h)$:
\begin{thm}[Montgomery]
\label{thm:SingularSeriesIntegerSum}
$$
\sum_{h=1}^H (\mathfrak{S}(h)-1)(1-\frac{h}{H}) \sim -\frac{1}{2}\log H.
$$
\end{thm}
Montgomery's proof of this theorem is recorded in \cite{MoSo1}, based upon an analysis of Dirichlet series and the observation that $\mathfrak{S}(h)$ is a constant factor times a multiplicative function.

It is worth remarking finally that for $H = X^{o(1)}$, one expects the Cram\'er model prediction \eqref{CramerModelIntegers} to be correct; and indeed for $H = \lambda \log X$, one does expect as predicted by the Cram\'er model that the count of primes $\pi(x;H)$ in a random short interval is distributed like a Poisson random variable. Gallagher \cite{Ga} showed that such a prediction is consistent with conjectures of Hardy and Littlewood which are slightly more general than Conjecture \ref{conj:HardyLittlewoodIntegers}. For $H$ growing faster than $\log X$, Montgomery and Soundararajan \cite{MoSo2} argued on the assumption of such Hardy-Littlewood conjectures that $\pi(x;H)$ tends towards a normal distribution, with variance given by \eqref{CramerModelIntegers} and \eqref{primevarianceZ}. For $H = X^{o(1)}$ this is consistent with the Cram\'er model, while for $H$ growing at a faster rate as we have explained it is not.

\subsection{Background and motivation: algebraic number fields}
\label{BackgroundNumberFields}

We will prove an analogue of Theorem \ref{thm:SingularSeriesIntegerSum} for arbitrary algebraic number fields and use this to develop evidence for a variant of Conjecture \ref{conj:primevariance}. We make use of a set up that was introduced by Tsai and Zaharescu in \cite{TsZa} to examine related questions.

We consider a number field $K$ with $[K: \mathbb{Q}] = n$ and choose a fixed integral basis $\{\alpha_1, ..., \alpha_n\}$ for $\mathcal{O}_K$, the ring of integers of $K$. Let $m: K \rightarrow \mathbb{Q}^n$ be the linear map defined with respect to such a basis by
\begin{equation}\label{mmap}
m(\alpha):= (k_1,...,k_n),\quad \textrm{for}\, \alpha = k_1 \alpha_1 + \cdots + k_n \alpha_n,\; \textrm{with}\; k_i \in \mathbb{Q}\; \textrm{for all}\; i.
\end{equation}
Note that $m$ takes $\mathcal{O}_K$ to $ \mathbb{Z}^n$.

Over $\mathbb{N}$ an interval of size $H$ near a point $x$ consisted of integers $n$ such that $x < n \leq x+H$. We generalize this to algebraic number fields by letting $\|\cdot\|$ be an arbitrary norm on $\mathbb{R}^n$, making $\mathbb{R}^n$ into a normed linear space. We will turn our attention in $\mathcal{O}_K$ to sets of the sort
$$
\{ \alpha \in \mathcal{O}_K:\, \|m(\alpha) - x \| \leq H\},
$$
as $x$ varies over $\mathbb{R}^n$. We will be interested in the variance of counts of prime elements in such `short intervals' chosen randomly. One motivation for this definition of a short interval comes from the work of Keating and Rudnick in a function field setting \cite{KeRu}.

Throughout this paper elements of $\mathcal{O}_K$ are denoted by Greek characters such as $\alpha$ and prime elements of this ring usually by $\omega$. (By ``prime element" we mean as usual that the ideal $(\omega)$ generated by this element is prime.) For ideals we use the script $\mathfrak{a}$, and for prime ideals usually $\mathfrak{p}$. We use $N(\alpha) = N_{K/\mathbb{Q}}(\alpha)$ to denote the norm of an element of $K$, while $N(\mathfrak{a})$ also denotes the norm of an ideal.

It follows from the generalized prime number theorem of Mitsui \cite[Main Theorem]{Mi} that
\begin{thm}[Mitsui's generalized prime number theorem]
\label{thm:Mitsui}
Let $K$ be an algebraic number field. For any norm $\|\cdot\|$ and map $m$ as above,
\begin{equation}
\label{generalPNT}
\# \{ \omega \in \mathcal{O}_K\; \textrm{prime}: \; \|m(\omega)\| \leq X\} \sim \frac{1}{\mathfrak{R}_K}\sum_{\substack{\|m(\alpha)\| \leq X \\ |N\alpha| > 1}} \frac{1}{\log |N(\alpha)|},
\end{equation}
where $\mathfrak{R}_K := \Res_{s=1} \zeta_K(s)$ is the residue at $s=1$ of the Dedekind zeta function of $K$, and the sum is over all $\alpha \in \mathcal{O}_K$.
\end{thm}
One should interpret \eqref{generalPNT} as saying that a random algebraic integer $\alpha$ has probability roughly $\frac{1}{\mathfrak{R}_K} \frac{1}{\log |N(\alpha)|}$ of being a prime element.

Theorem \ref{thm:Mitsui} generalizes the prime number theorem, and correspondingly Gross and Smith \cite{GrSm} have introduced an analogue of the Hardy-Littlewood Conjecture \ref{conj:HardyLittlewoodIntegers}.

\begin{conj}[Gross-Smith]
\label{conj:GrossSmithNumberFields}
Let $K$ be an algebraic number field and $\|\cdot\|$ a norm as above. For fixed nonzero $\eta \in \mathcal{O}_K$,
\begin{equation}
\label{GrossSmithNumberFields}
\sum_{\|m(\alpha)\| \leq X} \mathbf{1}_{\mathcal{P}}(\alpha) \mathbf{1}_\mathcal{P}(\alpha+\eta) \sim \mathfrak{S}(\eta) \frac{1}{\mathfrak{R}_K^2} \sum_{\substack{\|m(\alpha)\| \leq X \\ |N\alpha| > 1}} \frac{1}{\log^2 |N\alpha|},
\end{equation}
where
\begin{equation}
\label{singularseriesNumberFields}
\mathfrak{S}(\eta):= \prod_{\mathfrak{p}} \frac{(1 - \nu_\eta(\mathfrak{p})/N\mathfrak{p})}{(1-1/N\mathfrak{p})^2}, \quad \textrm{with}\quad \nu_\eta(\mathfrak{p}): =  \begin{cases} 2 & \textrm{if}\; \eta \notin \mathfrak{p} \\ 1 & \textrm{if} \; \eta\in\mathfrak{p} \end{cases}
\end{equation}
\end{conj}
Here both sums in \eqref{GrossSmithNumberFields} are over $\alpha \in \mathcal{O}_K$, and $\mathbf{1}_\mathcal{P}(\alpha)$ is $1$ if $\alpha$ is a prime element and $0$ otherwise. The product is over all prime ideals of $\mathcal{O}_K$.

In fact Gross and Smith make a more general conjecture than Conjecture \ref{conj:GrossSmithNumberFields}, just as Hardy and Littlewood did for the natural numbers. Conditioned on this more general conjecture, Tsai and Zaharescu show that the counts of prime elements in a random short interval $\{ \alpha \in \mathcal{O}_K: \|m(\alpha) - x \| \leq H\}$, where $x$ is chosen randomly and uniformly from the set $\{x \in \mathbb{R}^n: \|x\| \leq X\}$, tend to a Poisson random variable for $H = \lambda \log X$ as $X \rightarrow \infty$ for any fixed value of $\lambda$.\footnote{In fact, Tsai and Zaharescu work with a slightly more general notion of `short intervals' in number fields than we make use of here.} This is an analogue of the result of Gallagher over the integers, though Tsai and Zaharescu's proof necessarily depends on quite different ideas than Gallagher's at some points.

Finally, we make a note of \cite[Thm 4.1]{TsZa} of Tsai and Zaharescu:
\begin{thm}
\label{thm:logsum}
For $K$ an algebraic number field and $\|\cdot\|$ a norm as above,
$$
\sum_{\substack{\|m(\alpha)\| \leq X \\ |N\alpha|> 1}} \frac{1}{\log^k |N(\alpha)|} \sim \frac{\vol(B_X^{\|\cdot\|})}{\log^k \vol(B_X^{\|\cdot\|})},
$$
where $\vol(B_X^{\|\cdot\|})$ is the volume in $\mathbb{R}^n$ of the set $B_X^{\|\cdot\|} = \{x \in \mathbb{R}^n: \|x\|\leq X\}$.
\end{thm}
This can be used to simplify right hand side expressions of \eqref{generalPNT} and \eqref{GrossSmithNumberFields}.

\subsection{Main results}
\label{MainResults}

The main result proved in this paper is an estimate for sums of singular series in an algebraic number field. We use this result to heuristically derive a conjecture about the variance of primes in short intervals.

\begin{dfn}
\label{def:nice} 
A function $w: \mathbb{R}^n \rightarrow \mathbb{R}$ is called `nice' if $w$ has compact support and $\hat{w}(\xi) = O(1/|\xi|^{n+1})$.
\end{dfn}

A \emph{nice} function can thus be thought of as a function that is not too jagged. For instance, if $\mathcal{U}$ is a bounded open region of $\mathbb{R}^n$ with $C^\infty$-boundary and nonzero Gaussian curvature everywhere and $\mathbf{1}_\mathcal{U}$ is the indicator function of $\mathcal{U}$, then the convolution $\mathbf{1}_\mathcal{U}\ast\mathbf{1}_\mathcal{U}$ is nice \cite{Br,Sv}. 

Here and in what follows we use the convention that $\hat{w}(\xi) = \int_{\mathbb{R}^n} e(-x\cdot \xi) w(x)\, d^n x$, and $|\xi|$ is the Euclidean norm for $\xi \in \mathbb{R}^n$.

\begin{thm}
\label{thm:SingularSeriesGeneralSum}
For $K$ an algebraic number field with $[K:\mathbb{Q}] = n$, with an embedding $m: K \rightarrow \mathbb{R}^n$ as in section \ref{BackgroundNumberFields} and nice $w:\mathbb{R}^n\rightarrow\mathbb{R}$,
$$
\sum_{\substack{\eta \neq 0 \\ \eta \in \mathcal{O}_K}} (\mathfrak{S}(\eta)-1) w\left(\frac{m(\eta)}{H}\right) \sim - w(0) \mathfrak{R}_K \log(H^n).
$$
\end{thm}
As before $\mathfrak{R}_K = \Res_{s=1} \zeta_K(s)$.

We remark that probably Theorem \ref{thm:SingularSeriesGeneralSum} could be proved with more work for even a slightly wider class of test functions -- for instance we suspect Theorem \ref{thm:SingularSeriesGeneralSum} is true for $w = \mathbf{1}_\mathcal{U}\ast\mathbf{1}_\mathcal{U}$ for \emph{any} convex region $\mathcal{U}$ of $\mathbb{R}^n$.

The reader should take a moment to note that in any case this recovers Theorem \ref{thm:SingularSeriesIntegerSum} for $K=\mathbb{Q}$, using $w(x) = (1-|x|)_+$. The proof of Theorem \ref{thm:SingularSeriesGeneralSum} nonetheless is rather different from that given in \cite{MoSo1} for Theorem \ref{thm:SingularSeriesIntegerSum}. Rather than exploit the multiplicativity of $\mathfrak{S}(\eta)$ (a notion that makes sense only when $\mathcal{O}_K$ is a principal ideal domain), we expand $\mathfrak{S}(\eta)$ into Ramanujan sums and apply Fourier analysis. Such an approach was used in the second author's thesis \cite{Ro} to reprove Theorem \ref{thm:SingularSeriesIntegerSum}. New ideas are required in the passage to general number fields. 

We note that as a consequence of results proved by Tsai and Zaharescu (in particular see \cite[Thm. 1.3]{TsZa}) one may see that
$$
\sum_{\substack{\eta \neq 0 \\ \eta \in \mathcal{O}_K}} (\mathfrak{S}(\eta)-1) w\left(\frac{m(\eta)}{H}\right) = o(H^n).
$$
Our new contribution in this paper is a method to asymptotically identify the second order term of size $\log H$ above.

It is likely that a still lower-order constant term, depending upon the number field, can be identified in Theorem \ref{thm:SingularSeriesGeneralSum}, giving a formula with error $o(1)$. (Such a formula is found, for instance, over $\mathbb{Z}$ in \cite{MoSo1}.) We have not pursued these constant terms in this paper however.

We  use this computation of sums of singular series to provide heuristic evidence for a generalization of Conjecture \ref{conj:primevariance} on counts of primes in short intervals. We make use of the following set up, for a number field $K$, and an embedding $m:K \rightarrow \mathbb{R}^n$ and a norm $\|\cdot \|$ on $\mathbb{R}^n$ as before.

\begin{itemize}
\item Define
$$
\pi_K(x;H):= \# \{ \omega \in \mathcal{O}_K \; \textrm{prime}:\; \|m(\omega)-x\| \leq H\}.
$$
as a count of primes in a `short interval' around $x$.

\item Define the expectation of this quantity by
$$
E_K(X;H) := \frac{1}{\vol\left(B_X^{\|\cdot\|}\right)} \int_{\|x\|\leq X} \pi_K(x;H)\, d^nx,
$$
where, again, $\vol(B_X^{\|\cdot\|})$ is the volume in $\mathbb{R}^n$ of the set $B_X^{\|\cdot\|} = \{x\in\mathbb{R}^n:\, \|x\|\leq X\}$.

\item Note that Mitsui's generalized prime number theorem suggests $$\frac{1}{\mathfrak{R}_K}\sum_{\substack{\|m(\alpha)-x\|\leq H\\ |N(\alpha)| > 1}} \frac{1}{\log |N(\alpha)|}$$ is a good approximation for $\pi_K(x;H)$. Consider
$$
\widetilde{\pi}_K(x;H):= \pi_K(x;H) - \frac{1}{\mathfrak{R}_K}\sum_{\substack{\|m(\alpha)-x\|\leq H \\ |N(\alpha)| > 1}} 1/\log |N(\alpha)|,
$$
which keeps track of the deviation of $\pi_K(x;H)$ from this approximation.

\item Define 
$$
V_K(X;H) := \frac{1}{\vol\left(B_X^{\|\cdot\|}\right)} \int_{\|x\|\leq X} \widetilde{\pi}_K(x;H)^2 d^nx.
$$
As for the integers, this quantity is not exactly the probabilistic variance of the counts $\pi_K(x;H)$, but it should be thought of as being analogous.
\end{itemize}

For expectation it is easy to show from Mitsui's generalized prime number theorem and Theorem \ref{thm:logsum} that
\begin{prop}
\label{prop:ExpectedPrimes}
For $0 \leq H \leq X$,
$$
E_K(X;H) \sim \frac{\vol(B_H^{\|\cdot\|})}{\mathfrak{R}_K \log\,\vol(B_X^{\|\cdot\|})}.
$$
\end{prop}
Note that in this formula,
\begin{equation}
\label{exp_asymp}
\log\, \vol(B_X^{\|\cdot\|}) \sim \log(X^n),
\end{equation} 
but the other terms depend more delicately on the number field $K$ and the norm $\|\cdot\|$.

By contrast, we conjecture that the variance enjoys a simple and universal relation to expectation,
\begin{conj}
\label{conj:primevarianceNumberFields}
For any algebraic number field $K$ with $[K:\mathbb{Q}] = n$, with an embedding $m:\,K\rightarrow\mathbb{R}^n$ and norm $\|\cdot\|$ as above, for $\delta \in (0,1)$ and $H = X^\delta$, we have
$$
V_K(X;H) \sim (1-\delta)\, E_K(X;H).
$$
\end{conj}
We note that the parameter $\delta$ may be thought of somewhat more intrinsically as $\delta = (\log\, \vol(B_H^{\|\cdot\|}))/(\log\, \vol(B_X^{\|\cdot\|}))$.

A Cram\'er model for the distribution of primes elements in a number field would predict a statement of this sort with $(1-\delta)$ replaced by just $1$. Tsai and Zaharescu showed, conditionally on the Hardy-Littlewood type conjectures of Gross-Smith, that in short intervals with $H \approx \log X$, a Cram\'er model should accurately predict the distribution of $\pi_K(x;H)$. Conjecture \ref{conj:primevarianceNumberFields} suggests that for larger $H$ this is no longer the case. Moreover, a little curiously, this conjecture suggests that the variance deviates from the Cram\'er model prediction by the same universal factor $(1-\delta)$ no matter the number field $K$ one works in. We do not have a really simple conceptual explanation for this phenomenon except that it comes out of the computations to follow. Using Theorem \ref{thm:SingularSeriesGeneralSum} we give a heuristic derivation of Conjecture \ref{conj:primevarianceNumberFields} in section \ref{PrimesHeuristicDerivation}.

\subsection{Function fields}
We have mentioned that we were motivated to investigate Conjecture \ref{conj:primevarianceNumberFields} by work of Keating and Rudnick \cite{KeRu}, who proved a different analogue of Conjecture \ref{conj:GoldstonMontgomery} (or equivalently Conjecture \ref{conj:primevariance}) by replacing the integers with the ring of polynomials with coefficients in the finite field $\mathbb{F}_q$ and taking a large $q$ limit. It would be very interesting to see if a generalization of Keating and Rudnick's theorem, along the lines of Conjecture \ref{conj:primevarianceNumberFields}, can also be proved by considering function fields more general than the rational function field and making use of a suitable notion of a short interval in this setting.

\subsection{Acknowledgments}
\label{acknowledgements}

The second author received partial support from NSF grants DMS-1701577 and DMS-1854398 and an NSERC grant. We thank Ofir Gorodetsky, Robert Lemke Oliver, Jeff Lagarias, Jennifer Park, and Paul Pollack for helpful discussions.

\section{Some preliminary estimates in algebraic number theory}
\label{Preliminary}

\subsection{Ramanujan sums and other arithmetic functions in $\mathcal{O}_K$}
\label{RamanujanSum}
We will require some number field variants of some well known arithmetic functions. For ideals $\mathfrak{a}$ and $\mathfrak{c}$ we use the notation $\mathfrak{a}  | \mathfrak{c}$ as usual to mean $\mathfrak{c} \subseteq \mathfrak{a}$. Note that $\mathfrak{a} | \mathfrak{c}$ if and only if there exists an ideal $\mathfrak{b}$ such that $\mathfrak{ab} = \mathfrak{c}.$

For $K$ an algebraic number field and $\mathfrak{q}$ an ideal of $\mathcal{O}_K$ with prime factorization $\mathfrak{q} = \mathfrak{p}_1^{e_1}\cdots \mathfrak{p}_k^{e_k}$, define the M\"obius function on ideals by
$$
\mu(\mathfrak{q}) := \begin{cases} (-1)^k & \textrm{if}\; e_1 = \cdots = e_k = 1 \\ 0 & \; \textrm{if}\; e_i \geq 2\; \textrm{for some}\; i, \end{cases}
$$
and the totient function by
$$
\phi(\mathfrak{q}) = N\mathfrak{q} \prod_{\mathfrak{p} | \mathfrak{q}} \left(1-\frac{1}{N\mathfrak{p}}\right).
$$ 

Likewise for $\eta \in \mathcal{O}_K$ and $\mathfrak{q}$ still an ideal, define a \emph{Ramanujan sum} by
\begin{equation}
\label{Ramanujan_def}
c_\mathfrak{q}(\eta):= \sum_{\substack{\mathfrak{ab} = \mathfrak{q} \\ \eta \in \mathfrak{b}}} \mu(\mathfrak{a}) N\mathfrak{b}.
\end{equation}
As for the integers from M\"obius inversion for ideals $\mathfrak{c}$ and arbitrary $\eta$,
\begin{equation}
\label{Ramanujan_condense}
\sum_{\mathfrak{q}|\mathfrak{c}} c_{\mathfrak{q}}(\eta) = \begin{cases} N\mathfrak{c} & \textrm{if}\; \eta \in \mathfrak{c} \\
0 & \textrm{if}\; \eta \notin \mathfrak{c}.\end{cases}
\end{equation}

Note that for fixed $\eta$ the right hand side of \eqref{Ramanujan_condense} is multiplicative in ideals $\mathfrak{c}$; it follows as for the integers that $c_\mathfrak{q}(\eta)$ is multiplicative in $\mathfrak{q}$. More developed accounts of Ramanujan sums in algebraic number fields can be found in \cite{Ra,Ri}.

We make use of Ramanujan sums because, as over the integers, singular series such as $\mathfrak{S}(\eta)$ can be expressed elegantly in terms of them. In particular,
\begin{lem}
\label{lem:SingularToRamanujan}
For nonzero $\eta \in \mathcal{O}_K$, 
$$
\mathfrak{S}(\eta) = \sum_{\mathfrak{q}\neq 0} \frac{\mu(\mathfrak{q})^2}{\phi(\mathfrak{q})^2} c_\mathfrak{q}(\eta),
$$
where the sum is over all ideals $\mathfrak{q}$ of $\mathcal{O}_K$.
\end{lem}
\begin{proof}
The proof is the same as for the integers. The reader may verify that
$$
c_\mathfrak{p}(\eta) = \begin{cases} N\mathfrak{p} - 1 & \textrm{if}\; \eta \in \mathfrak{p} \\ -1 & \textrm{if}\; \eta \notin \mathfrak{p}\end{cases},
$$
and from this the Euler product \eqref{singularseriesNumberFields} can be written
$$
\mathfrak{S}(\eta) = \prod_\mathfrak{p} \left(1 + \frac{1}{\phi(\mathfrak{p})^2}c_\mathfrak{p}(\eta)\right).
$$
The lemma then follows by expanding the Euler product using the multiplicativity of the functions $\phi(\mathfrak{q})$ and $c_\mathfrak{q}(\eta)$. (Since $\eta \in \mathfrak{p}$ for only finitely many primes, there are no concerns about convergence.)
\end{proof}

We will later need to make recourse of the following estimates, whose proofs are exercises using Perron's formula and analytic properties of the zeta functions $\zeta_K(s)$. The reader may choose to skip these estimates for the moment, coming back to verify them when needed.

\begin{lem}
\label{lem:complicated_sums}
For $[K:\mathbb{Q}] = n$ we have as $Y\rightarrow\infty$,
\begin{equation}
\label{Asymptotic_sum}
\sum_{N\mathfrak{q} \leq Y} \frac{\mu^{2}(\mathfrak{q})}{\phi(\mathfrak{q})} = \mathfrak{R}_K \log(Y) + O(1),
\end{equation}
where $\mathfrak{R}_K:= \Res_{s=1}\zeta_K(s)$. 

Moreover,
\begin{equation}
\label{Sum_a}
\frac{1}{Y^{1/n}} \sum_{N\mathfrak{q} \leq Y} \frac{\mu(\mathfrak{q})^2}{\phi(\mathfrak{q})^2} (N\mathfrak{q})^{1+1/n} = O(1),
\end{equation}
\begin{equation}
\label{Sum_b}
\sum_{N\mathfrak{q} \geq Y} \frac{\mu(\mathfrak{q})^2}{\phi(\mathfrak{q})^2} \sum_{\substack{\mathfrak{b}| \mathfrak{q} \\ N\mathfrak{b} \leq Y}} N\mathfrak{b} = O(1),
\end{equation}
\begin{equation}
\label{Sum_c}
Y \sum_{N\mathfrak{q} \geq Y} \frac{\mu(\mathfrak{q})^2}{\phi(\mathfrak{q})^2} \sum_{\substack{\mathfrak{b} | \mathfrak{q} \\ N\mathfrak{b} \geq Y}} 1 = O(1),
\end{equation}
\begin{equation}
\label{Sum_d}
\frac{1}{Y^{1/n}} \sum_{N\mathfrak{q} \geq Y} \frac{\mu(q)^2}{\phi(q)^2} \sum_{\substack{\mathfrak{b}| \mathfrak{q} \\ N\mathfrak{b} \leq Y}} (N\mathfrak{b})^{1+1/n} = O(1),
\end{equation}
where the implicit constants depend upon the field $K$.
\end{lem}

\begin{proof} 
We treat each of \eqref{Asymptotic_sum},\eqref{Sum_a},\eqref{Sum_b},\eqref{Sum_c},\eqref{Sum_d} in turn. Our technique is the relatively standard complex analytic one and we make use of Perron's formula (see \cite[Prop. 5.54]{IwKa}), that for $c > 0$, $x > 0$
\begin{equation}
\label{perron}
\frac{1}{2\pi i} \int_{c-iT}^{c+iT} x^s \frac{ds}{s} = h(x) + O\left(\frac{x^c}{T|\log(x+2)|}\right), \quad \textrm{where} \quad h(x) = \begin{cases} 0 & \textrm{if}\; 0 < x < 1 \\ 1/2 & \textrm{if}\; x=1 \\ 1 & \textrm{if}\; x>1 \end{cases},
\end{equation}
along with a slightly smoothed variant (see \cite[Lem. 3, Ch. 13]{Ap}), that for $c > 0$, $u > 0$,
\begin{equation}
\label{perron2}
\frac{1}{2\pi i} \int_{c-i\infty}^{c+i\infty} \frac{u^{-s}}{s(s+1)} \, ds = I(u), \quad \textrm{where} \quad I(u) = (1-u)_+ = \begin{cases} (1-u) & \textrm{if}\; 0 < u \leq 1 \\ 0 & \textrm{if}\;  u > 1\end{cases}.
\end{equation}

We also make use of the following convexity estimate: for fixed number field $K$ with $[K:\mathbb{Q}] = n$, for $0 \leq \sigma \leq 1$ and $|t| \geq 1$,
\begin{equation}
\label{convexity}
\zeta_K(s) \ll_\epsilon |s|^{n(1-\sigma)/2 + \epsilon},
\end{equation}
for any $\epsilon > 0$. (This follows directly from \cite[Thm. 4]{Ra2}.)

\medskip
\textit{To verify \eqref{Asymptotic_sum}:} note that for $\Re s > 1$,
$$
\sum_\mathfrak{q} \frac{\mu(\q)^2}{\phi(\q)(N\q)^s} = \prod_\p \left(1 + \frac{1}{(N\p-1)(N\p)^s}\right) = \zeta_K(s+1)F(s),
$$
for
$$
F(s):= \prod_\p \left(1 - \frac{1}{(N\p)^{s+1}}
\right)\left(1+ \frac{1}{(N\p-1)(N\p)^s}
\right),
$$
which converges and is bounded in the region $\Re s > -1/2 + \epsilon$ for any $\epsilon > 0$, and moreover $F(0)=1$.

We evaluate
\begin{equation}
\label{Mellin1}
\frac{1}{2\pi i} \int_{1/n - i Y^{2/n}}^{1/n + i Y^{2/n}} \sum_\q \frac{\mu(\q)^2}{\phi(\q) (N\q)^s} \frac{Y^s}{s}\,ds
\end{equation}
in two different ways. In the first place by \eqref{perron} we have
$$
\eqref{Mellin1} = \sum_\q \frac{\mu(\q)^2}{\phi(\q)}\Big[ h\left(\frac{Y}{N\q}
\right) + O\left(\frac{(Y/N\q)^{1/n}}{Y^{2/n}}
\right)\Big] = \sum_{N\q < Y} \frac{\mu(\q)^2}{\phi(\q)} + O(1).
$$
But shifting the contour in \eqref{Mellin1} to the left by $2/n$ units, we see
\begin{multline*}
\eqref{Mellin1} = \Res_{s=0} \zeta_K(s+1) F(s) \frac{Y^s}{s} + O(1) + \int_{-1/n - i Y^{2/n}}^{-1/n+ i Y^{2/n}} \zeta_K(s+1) F(s) \frac{Y^s}{s}\, ds \\
= \mathfrak{R}_K \log Y + O(1) + O\left(Y^{-1/n}\int_1^{Y^{2/n}} \frac{t^{n/2n}}{t}\, dt
\right),
\end{multline*}
using \eqref{convexity} to bound terms arising from the contour shift. Simplifying the error term by evaluating the integral, and pairing up the two expressions for \eqref{Mellin1} gives \eqref{Asymptotic_sum}.

\medskip
\textit{To verify \eqref{Sum_a}:} we take a similar strategy. Note that for $\Re s > 1/n$,
$$
\sum_\q \frac{\mu(\q)^2}{\phi(\q)^2} \frac{(N\q)^{1+1/n}}{(N\q)^s} = \prod_\p\left(1 + \frac{(N\p)^{1+1/n-s}}{(N\p-1)^2}
\right) = \zeta_K(1+s-1/n) G_1(s),
$$
where $G(s)$, using the same idea as above, is a function that converges and is bounded in the region $\Re s > -1/2 + \epsilon$ for any $\epsilon > 0$. Yet
\begin{align*}
\sum_{N\q \leq Y} \frac{\mu(\q)^2}{\phi(\q)^2} (N\q)^{1+1/n} &\ll \sum_\q \frac{\mu(\q)^2}{\phi(\q)^2} (N\q)^{1+1/n} I\left(\frac{N\q}{2Y}
\right) \\
&= \frac{1}{2\pi i} \int_{2/n-i\infty}^{2/n+i\infty} \zeta_K(1+s-1/n) G_1(s) \frac{(2Y)^{s}}{s(s+1)}\, ds \\
&= \Res_{s=1/n} \zeta_K(1+s-1/n) G_1(s) \frac{(2Y)^{s}}{s(s+1)} + O(Y^{1/2n}),
\end{align*}
with the last step following by shifting the contour to the line $\Re s = 1/2n$, using \eqref{convexity} to trivially bound contributions from the contour shift. The residue in the final line produces the bound of $O(Y^{1/n})$ from which \eqref{Sum_a} follows.

\medskip
\textit{To verify \eqref{Sum_b}:} note that for $\Re(s_2-s_1) > 0$ and $\Re s_1 < 1$,
\begin{align*}
\sum_\q \frac{\mu(\q)^2}{\phi(\q)^2} (N\q)^{s_1} \sum_{\bee | \q} (N\bee)^{1-s_2} &= \prod_\p \left(1+ \frac{(N\p)^{s_1}}{(N\p-1)^2} (1+ (N\p)^{1-s_2})
\right) \\
&= \zeta_K(1+s_2-s_1) G_2(s_1,s_2),
\end{align*}
where $G_2$ is defined by the Euler product
\begin{equation}
\label{G_2}
G_s(s_1,s_2) := \prod_\p \left(1 - \frac{1}{(N\p)^{1+s_2-s_1}}
\right) \left( 1+ \frac{(N\p)^{s_1}}{(N\p-1)^2} (1 + (N\p)^{1-s_2})
\right),
\end{equation}
which the reader should check converges and is bounded for $s_1, s_2$ in the region
$$
\Re s_1 \leq 1-\epsilon, \quad \Re s_2 \geq \Re s_1 -1/2 +\epsilon,
$$
for any $\epsilon > 0$. Note that
\begin{align}
\label{Mellin3}
\sum_{N\q \geq Y} \frac{\mu(\q)^2}{\phi(\q)^2} \sum_{\substack{\bee | \q \\ N\bee \leq Y}} N\bee 
& \ll \sum_\q \frac{\mu(\q)^2}{\phi(\q)^2} I\left(\frac{Y}{2 N\q}
\right) \sum_{\bee | \q} N\bee \cdot I\left(\frac{N\bee}{2Y}
\right) \\
\notag  &= \frac{1}{(2\pi i)^2} \int_{1/n - i \infty}^{1/n + i\infty} \int_{2/n-i\infty}^{2/n + i\infty} \zeta_K(1+s_2-s_1) G_2(s_1,s_2)\\ 
&\notag \hspace{45mm} \times \frac{(Y/2)^{-s_1}}{s_1(s_1+1)} \frac{(2Y)^{s_2}}{s_2(s_2+1)} \,ds_2\, ds_1.
\end{align}
Carefully shifting the $s_2$ contour to $\Re s_2 = 1/2n$, again using \eqref{convexity} to bound errors, we see that \eqref{Mellin3} is equal to
\begin{align*}
& \frac{1}{2\pi i} \int_{1/n - i\infty}^{1/n + i \infty} \Res_{s_2=s_1} \left(\zeta_K(1+s_2-s_1) G_2(s_1,s_2) \frac{(Y/2)^{-s_1}}{s_1(s_1+1)} \frac{(2Y)^{s_2}}{s_2(s_2+1)}
\right)\, ds_1 + O(Y^{-1/2n}) \\
&=\int_{1/n-i \infty}^{1/n+i\infty} O\left(\frac{1}{|s_1|^2|s_1+1|^2}
\right)\, |ds_1| + O(Y^{-1/2n}) \\
&= O(1),
\end{align*}
which is \eqref{Sum_b}.

\medskip
\textit{To verify \eqref{Sum_c}:} note that for $\Re (s_1 + s_2) < 1$ and $\Re s_1 < 1$,
$$
\sum_\q \frac{\mu(\q)^2}{\phi(\q)^2} (N\q)^{s_1} \sum_{\bee | \q} (N\bee)^{s_2} = \zeta_K(2-s_1-s_2) G_2(s_1,1-s_2),
$$
for $G_2$ defined in \eqref{G_2}; thus $G_2(s_1,1-s_2)$ is bounded for $s_1, s_2$ in the region $s_1,s_2$ with $\Re s_1 \leq 1-\epsilon$ and $\Re (s_1+s_2) \leq 3/2-\epsilon$. Moreover, much like above,
\begin{align}
\label{Mellin4}
\sum_{N\q \geq Y} \frac{\mu(\q)^2}{\phi(\q)^2} \sum_{\substack{ \bee | \q \\ N\bee \geq Y}} 1 &\ll \sum_q \frac{\mu(\q)^2}{\phi(\q)^2} I\left(\frac{Y}{2N\q}
\right) \sum_{\bee | \q} I\left(\frac{Y}{2 N\bee}
\right) \\
\notag &= \frac{1}{(2\pi i)^2} \int_{1/2 - i\infty}^{1/2+i \infty} \int_{1/2-1/n - i \infty}^{1/2 - 1/n + i\infty} \zeta_K(2-s_1-s_2) G_2(s_1, 1-s_2) \\
\notag &\hspace{50mm}\times  \frac{(Y/2)^{-s_1}}{s_1(s_1+1)} \frac{(Y/2)^{-s_2}}{s_2(s_2+1)}\, ds_2\, ds_1.
\end{align}
Carefully shifting the $s_2$ contour to $\Re s_2 = 1/2 + 1/n$, once more using \eqref{convexity} to bound error terms arising from the contour shift, we have \eqref{Mellin4} is equal to
\begin{align*}
&\frac{1}{2\pi i} \int_{1/2-i \infty}^{1/2+i\infty} \left( \Res_{s_2 = 1-s_1} \zeta_K(2-s_1-s_2) G_2(s_1, 1-s_2) \frac{(Y/2)^{-s_1}}{s_1(s_1+1)} \frac{(Y/2)^{-s_2}}{s_2(s_2+1)}
\right)\, ds_1 \\
&\hspace{5mm} + O(Y^{-(1+1/n)}) \\
&= \int_{1/2-i\infty}^{1/2+i\infty} O(Y^{-1}/|s_1|^4)\, |ds_1| + O(Y^{-(1+1/n)}) \\
& = O(Y^{-1}),
\end{align*}
from which \eqref{Sum_c} follows.

\medskip
\textit{To verify \eqref{Sum_d}:} note that it is a direct consequence of \eqref{Sum_b}, because
$$
\frac{1}{Y^{1/n}} \sum_{N\q \geq Y} \frac{\mu(\q)^2}{\phi(\q)^2} \sum_{\substack{\bee | \q \\ N\bee \leq Y}} (N\bee)^{1+1/n} \leq \sum_{N\q \geq Y} \frac{\mu(\q)^2}{\phi(\q)^2} \sum_{\substack{\bee | \q \\ N\bee \leq Y}} N\bee.
$$
\end{proof}

\subsection{Counts of ideal elements}
\label{IdealCounts}

We will need a few estimates regarding the distribution of lattices induced by ideals. 
\begin{lem}
\label{lem:not_too_skew}
For $[K:\mathbb{Q}] = n$ and an ideal $\mathfrak{a} \subset \mathcal{O}_K$, define the lattice $\Lambda_\mathfrak{a} = \{m(\alpha):\; \alpha \in \mathfrak{a}\} \subset \mathbb{Z}^n$ and the dual lattice $\Lambda_\mathfrak{a}^\ast := \{y \in \mathbb{R}^n:\; \textrm{for all}\; x\in \Lambda_\mathfrak{a},\; \langle x,y \rangle \in \mathbb{Z}\}$. Finally define the count
$$
C(\Lambda_\mathfrak{a}^\ast, r) := \#\{y \in \Lambda_\mathfrak{a}^\ast:\; y \neq 0, \, |y|\leq r\}.
$$
Then we have a bound
$$
C(\Lambda_\mathfrak{a}^\ast, r) \ll_{m,K} N\mathfrak{a} \cdot r^n.
$$
\end{lem}

Note that for any lattice $\Lambda \subset \mathbb{R}^n$, for sufficiently large $r$, $\#\{x \in \Lambda^\ast:\; x\neq 0,\, |x|\leq r\} \sim \det(\Lambda) \vol(B^r)$, where $\det(\Lambda)$ is the volume of the fundamental domain of $\Lambda$ and $B_r$ is the Euclidean ball of radius $r$. Thus the theorem is made significant by being a bound \emph{uniform} in $\mathfrak{a}$ and $r$.

In particular, this lemma implies that for $r \ll 1/(N\mathfrak{a})^{1/n}$, there are no non-zero elements of $\Lambda_\mathfrak{a}^\ast$ in $B_r$. This is a feature of lattices that are not too distorted, and it is exactly this feature of the lattice $\Lambda_\mathfrak{a}$ that we will make use of. This fact and by implication Lemma 3 seem to be well known (see e.g. \cite{Ka} or \cite{CaHaLeOlPoTh}), but we indicate the proof below.

In our proof and in what follows we make use of the Minkowski embedding $M: K \rightarrow \mathbb{R}^n$, defined by
$$
M(\alpha) = \left(\sigma_1(\alpha),...,\sigma_r(\alpha), \Re \tau_1(\alpha),...,\Re \tau_s(\alpha), \Re \tau_1(\alpha), ..., \Re \tau_s(\alpha)
\right),
$$
where $\sigma_j$ are the embeddings $K\rightarrow\mathbb{R}$ and $\tau_j$ are the complex embeddings $K\rightarrow\mathbb{C}$, with one chosen from each conjugate pair (so $r+2s = n$). $M$ is one to one and linear over $\mathbb{Q}$; it follows that for the map $m: K\rightarrow \mathbb{R}^n$ defined in $\eqref{mmap}$, $m = TM$ for some invertible matrix $T : \mathbb{R}^n\rightarrow\mathbb{R}^n$ ($T$ depends of course on the choice of basis elements used to define $m$).

\begin{proof}[Proof of Lemma \ref{lem:not_too_skew}]
We rely upon the following claim (see e.g. (83) of \cite{Ri2}) \emph{For fixed $K$, there exists a constant $c_K$ such that for any ideal $\mathfrak{a} \subset \mathcal{O}_K$, there exists an integral basis $\eta_1, ..., \eta_n$ such that
$$
\| M(\eta_i)\|_{\ell^\infty} \leq c_K (N\mathfrak{a})^{1/n} \quad \textrm{for all}\; i.
$$}

As $m = TM$ for an invertible map $T$, one also has
$$
\| m(\eta_i)\|_{\ell^\infty} \ll_{m,K} (N\mathfrak{a})^{1/n} \quad \textrm{for all}\; i.
$$
The lattice $\Lambda_\mathfrak{a}$ is generated by $m(\eta_1),..., m(\eta_n)$, so that for the $n\times n$ matrix $L_\eta$ with these vectors as columns,
$$
L_\eta:= \begin{pmatrix} | & & | \\ m(\eta_1) & \cdots & m(\eta_n) \\ | & & | \end{pmatrix},
$$
we have $\Lambda_\mathfrak{a} = L_\eta(\mathbb{Z}^n)$. Note that $\Lambda_\mathfrak{a}^\ast = L^{-1}_\eta(\mathbb{Z}^n),$ so that
\begin{align*}
C(\Lambda_\mathfrak{a}^\ast, r) &= \#\{ w \in \mathbb{Z}^n:\; w \neq 0, \, |L_\eta^{-1}w| \leq r\} \\
& = \# (\mathbb{Z}^n\setminus \{0\})\cap L_\eta(B_r).
\end{align*}
But because all the columns of $L_\eta$ are bound in norm by $(N\mathfrak{a})^{1/n}$, we have
$$
L_\eta(B_r) \subset B_{(c_{m,K,n} (N\mathfrak{a})^{1/n}r)},
$$
for some constant $c_{m,K,n}$ that depends only upon $m, K$ and the dimension $n$ of the space in which we have embedded (and thus only upon $m$ and $K$). But it is classical that $ \# (\mathbb{Z}^n\setminus \{0\})\cap B_s \ll_n s^n$ for any $s > 0$, so
$$
\# (\mathbb{Z}^n\setminus \{0\})\cap B_{(c_{m,K,n} (N\mathfrak{a})^{1/n}r)} \ll N\mathfrak{a} r^n,
$$
which establishes the claimed bound.
\end{proof}

\begin{lem}
\label{lem:smooth_counts}
Fix $[K:\mathbb{Q}]=n$, nice $w: \mathbb{R}^n \rightarrow \mathbb{R}$, and an embedding $m: K\rightarrow \mathbb{R}^n$ as in $\eqref{mmap}$. Then there exists a constant $C$ such that for any ideal $\mathfrak{a}$ and all $H >0$,
\begin{equation}
\label{smooth_counts}
\sum_{\eta \in \mathfrak{a}} w\left(\frac{m(\eta)}{H}
\right) = \begin{cases} w(0) & \textrm{for}\; N\mathfrak{a} \geq C H^n \\ \frac{H^n}{N\mathfrak{a}} \hat{w}(0) + O\left(\frac{(N\mathfrak{a})^{1/n}}{H}
\right) & \textrm{for}\; N\mathfrak{a} < CH^n.\end{cases}
\end{equation}
\end{lem}

Recall that \emph{nice} functions were defined in Definition \ref{def:nice}.

\begin{proof}
We verify \eqref{smooth_counts} for $N\mathfrak{a} \geq CH^n$ first. Because $m = TM$ for invertible $T$,
$$
\Big| \frac{m(\eta)}{H} \Big| \gg \Big| \frac{M(\eta)}{H}\Big|,
$$
But using the arithmetic geometric mean inequality to proceed to the second line below, we have,
\begin{multline*}
|M(\eta)| \gg (|\sigma_1(\eta)|^2 + \cdots + \sigma_r(\eta) + 2|\tau_1(\eta)|^2 + \cdots + 2|\tau_s(\eta)|^2)^{1/2} \\ \gg (|\sigma_1(\eta)\cdots\sigma_r(\eta)| |\tau_1(\eta)|^2\cdots |\tau_s(\eta)|^2)^{1/n} = |N\eta|^{1/n}.
\end{multline*}
But for nonzero $\eta \in \mathfrak{a}$, we have $|N\eta| \geq N\mathfrak{a}$. Collecting the above observations we have for nonzero $\eta \in \mathfrak{a}$,
$$
\Big|\frac{m(\eta)}{H}\Big| \gg \frac{(N\mathfrak{a})^{1/n}}{H}.
$$
Hence if $C$ is sufficiently large that the function $w$ is supported in a ball of radius $\ll C^{1/n}$, the top line of \eqref{smooth_counts} will be true, since the only nonzero term in the sum will arise from $\eta = 0$.

We now turn to a proof of \eqref{smooth_counts} in the case $N\mathfrak{a} < CH^n$. In fact the estimate in this second half of \eqref{smooth_counts} will always be true; it ceases however to be a good bound when $N\mathfrak{a} \geq CH^n$.

Note that
\begin{align}
\label{fourier_lattice}
\notag \sum_{\eta\in\mathfrak{a}} w\left(\frac{m(\eta)}{H}
\right) &= \sum_{x \in \Lambda_\mathfrak{a}} w(x/H) \\
\notag &= \frac{H^n}{\det \Lambda_\mathfrak{a}} \sum_{y \in \Lambda_\mathfrak{a}^\ast} \hat{w}(Hy) \\
&= \frac{H^n}{\det \Lambda_\mathfrak{a}} \hat{w}(0) + \frac{H^n}{\det \Lambda_\mathfrak{a}} \sum_{\substack{ y \in \Lambda_\mathfrak{a}^\ast \\ y \neq 0}} \hat{w}(Hy).
\end{align}
Furthermore note that here
\begin{equation}
\label{det_norm}
\det \Lambda_\mathfrak{a} = N\mathfrak{a}.
\end{equation}
We have $\hat{w}(\xi) = O(1/|\xi|^{n+1})$ by hypothesis, and so by taking a dyadic decomposition we have,
\begin{equation}
\label{dyadic}
\sum_{\substack{ y \in \Lambda_\mathfrak{a}^\ast \\ y \neq 0}} \hat{w}(Hy) \ll \sum_{\nu = 0}^\infty \frac{1}{(2^\nu)^{n+1}} C(\Lambda_\mathfrak{a}^\ast, 2^\nu/H).
\end{equation}
As Lemma \ref{lem:not_too_skew} implies that for $N\mathfrak{a} \cdot (2^\nu/H)^n \ll 1$ we have $C(\Lambda_\mathfrak{a}^\ast, 2^\nu/H) = 0$, and as Lemma \ref{lem:not_too_skew} implies the bound $C(\Lambda_\mathfrak{a}^\ast, 2^\nu/H) \ll N\mathfrak{a}(2^\nu/H)^n$ otherwise, we can bound \eqref{dyadic} by
$$
\ll \sum_{2^\nu/H \gg (N\mathfrak{a})^{-1/n}} \frac{1}{(2^\nu)^{n+1}} N\mathfrak{a}\cdot (2^\nu/H)^n \ll \frac{(N\mathfrak{a})^{1+1/n}}{H^{n+1}}.
$$
Using this estimate in \eqref{fourier_lattice} with \eqref{det_norm}, we obtain the second line of \eqref{smooth_counts}.
\end{proof}

\section{Sums of singular series: a proof of Theorem \ref{thm:SingularSeriesGeneralSum}}
\label{MainTheoremProof}

We now turn in earnest to the proof of Theorem \ref{thm:SingularSeriesGeneralSum}.

\begin{proof}[Proof of Theorem \ref{thm:SingularSeriesGeneralSum}]
We proceed in three steps.

\medskip
\emph{Step 1:} Note from Lemma \ref{lem:SingularToRamanujan}, for $\eta \neq 0$,
$$
\mathfrak{S}(\eta) -1 = \sum_{N\mathfrak{q} > 1} \frac{\mu(\mathfrak{q})^2}{\phi(\mathfrak{q})^2} c_\mathfrak{q}(\eta).
$$
Hence using that $c_\mathfrak{q}(0) = \phi(\mathfrak{q})$, we have
\begin{equation}
\label{oscillatory_parts}
\sum_{\substack{\eta \neq 0 \\ \eta \in \mathcal{O}_K}} (\mathfrak{S}(h)-1) w\left(\frac{m(\eta)}{H}
\right) = \sum_{N\mathfrak{q} > 1} \frac{\mu(\mathfrak{q})^2}{\phi(\mathfrak{q})^2} \Big[ \sum_{\eta \in \mathcal{O}_k} c_\mathfrak{q}(\eta) w\left(\frac{m(\eta)}{H}
\right) - \phi(\mathfrak{q}) w(0)\Big].
\end{equation}
The advantage of \eqref{oscillatory_parts} is that we will be able to effectively control the sums of Ramanujan sums in the inner brackets.

\medskip
\emph{Step 2:} We define
$$
S_\mathfrak{q}(H) := \sum_{\eta \in \mathcal{O}_K} c_\mathfrak{q}(\eta) w\left(\frac{m(\eta)}{H}
\right),
$$
and give a good estimate for such sums. Note that
\begin{equation}
\label{sum_of_Sq}
\sum_{\mathfrak{q} | \mathfrak{a}} S_\mathfrak{q}(H) = \sum_{\eta \in \mathfrak{a}} w\left(\frac{m(\eta)}{H}
\right) N\mathfrak{a},
\end{equation}
by \eqref{Ramanujan_condense}. But the quantity on the right of this equation is estimated by Lemma \ref{lem:smooth_counts}. Hence using M\"obius inversion on \eqref{sum_of_Sq} followed by this Lemma we have,
\begin{align*}
S_\mathfrak{q}(H) &= \sum_{\mathfrak{a}\mathfrak{b} = \mathfrak{q}} \mu(\mathfrak{a}) \left( \sum_{\eta \in \mathfrak{b}} w\left(\frac{m(\eta)}{H}
\right) N\mathfrak{b}
\right) \\
& =\sum_{\substack{\mathfrak{ab} = \mathfrak{q} \\ N\mathfrak{b} \geq C H^n}} \mu(\mathfrak{a}) N\mathfrak{b}\cdot  w(0) + \sum_{\substack{\mathfrak{ab} = \mathfrak{q} \\ N\mathfrak{b} < CH^n}}\bigg( \mu(\mathfrak{a}) H^n \hat{w}(0) + O\left(\frac{(N\mathfrak{b})^{1/+1/n}}{H}
\right)\bigg). 
\end{align*}
This gives for $N\mathfrak{q} < CH^n$,
\begin{equation}
\label{Sq_est1}
S_\mathfrak{q}(H) = O\left(\frac{1}{H} \sum_{\mathfrak{b}|\mathfrak{q}} (N\mathfrak{b})^{1+1/n}
\right) = O\left(\frac{(N\mathfrak{q})^{1+1/n}}{H}
\right)
\end{equation}
and for $N\mathfrak{q} \geq CH^n$,
\begin{multline}
\label{Sq_est2}
S_\mathfrak{q}(H) = \phi(\mathfrak{q}) w(0) + O\left(\sum_{\substack{\mathfrak{b} | \mathfrak{q} \\ N\mathfrak{b} < CH^n}} N\mathfrak{b}
\right) \\+ O\left(H^n \sum_{\substack{\mathfrak{b} | \mathfrak{q} \\ N\mathfrak{b} \geq CH^n}} 1
\right) + O\left( \frac{1}{H} \sum_{\substack{\mathfrak{b} | \mathfrak{q} \\ N\mathfrak{b} < CH^n}} (N\mathfrak{b})^{1+1/n}
\right).
\end{multline}

\medskip
\emph{Step 3:} Substituting \eqref{Sq_est1}, \eqref{Sq_est2} into the equation \eqref{oscillatory_parts} derived in Step 1, we obtain
\begin{align*}
\sum_{\substack{\eta \neq 0 \\ \eta \in \mathcal{O}_K}} (\mathfrak{S}(h)-1) w\left(\frac{m(\eta)}{H}
\right) =& \sum_{1 < N\mathfrak{q} < CH^n} \frac{\mu(\mathfrak{q})^2}{\phi(\mathfrak{q})^2}\Big[-\phi(\mathfrak{q}) w(0) + O\left(\frac{(N\mathfrak{q})^{1+1/n}}{H}
\right)\Big] \\
& + \sum_{N\mathfrak{q} \geq C H^n} \frac{\mu(\mathfrak{q})^2}{\phi(\mathfrak{q})^2}\Big[ O\left(\sum_{\substack{\mathfrak{b} | \mathfrak{q} \\ N\mathfrak{b} < CH^n}} N\mathfrak{b}
\right)  + O\left(H^n \sum_{\substack{\mathfrak{b} | \mathfrak{q} \\ N\mathfrak{b} \geq CH^n}} 1
\right) \\ 
&\hspace{40mm} + O\left( \frac{1}{H} \sum_{\substack{\mathfrak{b} | \mathfrak{q} \\ N\mathfrak{b} < CH^n}} (N\mathfrak{b})^{1+1/n}
\right)\Big]
\end{align*}
Each error term being summed here can be controlled using \eqref{Sum_a},\eqref{Sum_b},\eqref{Sum_c},\eqref{Sum_d} of Lemma \ref{lem:complicated_sums} (with $Y = CH^n$), and an asymptotic formula for the first (main) term is given by \eqref{Asymptotic_sum}. 

Simplifying, we obtain
$$
\sum_{\substack{\eta \neq 0 \\ \eta \in \mathcal{O}_K}} (\mathfrak{S}(h)-1) w\left(\frac{m(\eta)}{H}
\right) = - w(0) \mathfrak{R}_K \log(H^n) + O(1),
$$
as claimed.
\end{proof}

\section{Primes in short intervals: a heuristic argument}
\label{PrimesHeuristicDerivation}

We now turn to a \emph{heuristic} justification of Conjecture \ref{conj:primevarianceNumberFields}, based on the computation in Theorem \ref{thm:SingularSeriesGeneralSum}. We do not formalize this heuristic; the best that one could hope for, without a real breakthrough, is a rather restricted conditional implication: that a version of Conjecture \ref{conj:GrossSmithNumberFields} with uniform error terms in $\eta$ implies Conjecture \ref{conj:primevarianceNumberFields} for a restricted range of $H$. In the heuristic that follows, in addition to assuming Conjecture \ref{conj:GrossSmithNumberFields}, we will also assume that the error terms between the right and left hand sides of \eqref{GrossSmithNumberFields} exhibit significant cancellation when summed over $\eta$, and we will not be careful in controlling the boundary-portion of certain sums that appear below. Finally, owing to the restriction of test functions $w$ in Theorem \ref{thm:SingularSeriesGeneralSum}, our justification below will really only be complete for norms $\|\cdot\|$ defined with respect to a convex body with $C^\infty$-boundary and nonzero Gaussian curvature everywhere; we believe however that Conjecture \ref{conj:primevarianceNumberFields} is true in the generality stated. Lines below that indicate an asymptotic relation derived heuristically rather than rigorously we demarcate using  `$\approx$'.

\medskip

\noindent \emph{Heuristic derivation of Conjecture \ref{conj:primevarianceNumberFields}.} 
By expanding the variance we have,
\begin{align*}
\Var_K(X;H) \approx& \;\frac{1}{\vol(B_X^{\|\cdot\|})} \int_{\|t\|\leq X} \bigg|\sum_{\substack{\alpha \in \mathcal{O}_K \\ \|m(\alpha)-t\| \leq H}} \mathbf{1}_\mathcal{P}(\alpha)\bigg|^2\, dt - E_K(X;H)^2 \\
=& \;\frac{1}{\vol(B_X^{\|\cdot\|})} \sum_{\alpha \in \mathcal{O}_K} \mathbf{1}_\mathcal{P}(\alpha)^2 \int_{\substack{\|t\|\leq X \\ \|t - m(\alpha) \| \leq H}}dt \\
&\; + \frac{1}{\vol(B_X^{\|\cdot\|})} \sum_{\substack{\alpha_1,\alpha_2 \in \mathcal{O}_K \\ \alpha_1\neq \alpha_2}} \mathbf{1}_\mathcal{P}(\alpha_1) \mathbf{1}_\mathcal{P}(\alpha_2) \int_{\substack{\|t\|\leq X \\ \|t-m(\alpha_1)\| \leq H \\ \|t-m(\alpha_2)\| \leq H}} dt - E_K(X;H)^2\\
\approx&\; E_K(X;H) + \frac{H^n}{\vol(B_X^{\|\cdot\|})} \sum_{\|m(\alpha)\|\leq X} \sum_{\substack{\eta \in \mathcal{O}_K \\ \eta \neq 0}} \mathbf{1}_\mathcal{P}(\alpha) \mathbf{1}_\mathcal{P}(\alpha+\eta) w\left(\frac{m(\eta)}{H}
\right) \\
&- E_K(X;H)^2,
\end{align*}
where $w(x) := \mathbf{1}_\mathcal{U}\ast \mathbf{1}_\mathcal{U}(x)$ and $\mathcal{U} = \{x\in \mathbb{R}^n: \|x\|\leq 1\}$.

Utilizing Conjecture \ref{conj:GrossSmithNumberFields} (along with Theorem \ref{thm:logsum}), it is reasonable to suppose the quantity above is approximated by
\begin{equation}
\label{approx_expansion}
\approx E_K(X;H) + \frac{H^n}{\mathfrak{R}_K^2 \log^2 \vol(B_X^{\|\cdot\|})} \sum_{\substack{\eta \in \mathcal{O}_K \\ \eta \neq 0}} \mathfrak{S}(\eta) w\left(\frac{m(\eta)}{H}
\right) - E_K(X;H)^2.
\end{equation}
But approximating the lattice sum below with an integral and using Proposition \ref{prop:ExpectedPrimes} we have,
\begin{align*}
\frac{H^n}{\mathfrak{R}_K^2 \log^2 \vol(B_X^{\|\cdot\|})} \sum_{\substack{\eta \in \mathcal{O}_K \\ \eta \neq 0}} w\left(\frac{m(\eta)}{H}
\right) & \sim \frac{H^n}{\mathfrak{R}_K^2 \log^2 \vol(B_X^{\|\cdot\|})} \int \mathbf{1}_\mathcal{U}\ast\mathbf{1}_\mathcal{U}\left(\frac{x}{H}
\right)\, d^nx \\
&= \frac{H^{2n} \vol^2(\mathcal{U}) }{\mathfrak{R}_K^2 \log^2 \vol(B_X^{\|\cdot\|})} \\
&\sim E_K(X;H)^2.
\end{align*}
Substituting the above for $E_K(X;H)^2$, we therefore have that \eqref{approx_expansion} is reasonably approximated by
\begin{align*}
&\approx E_K(X;H) + \frac{H^n}{\mathfrak{R}_K^2 \log^2 \vol(B_X^{\|\cdot\|})} \sum_{\substack{\eta \in \mathcal{O}_K \\ \eta \neq 0}} (\mathfrak{S}(\eta)-1) w\left(\frac{m(\eta)}{H}
\right) \\
&= E_K(X;H) + (1+o(1)) \frac{H^n}{\mathfrak{R}_K^2 \log^2 \vol(B_X^{\|\cdot\|})} \cdot \left(- \mathbf{1}_\mathcal{U}\ast\mathbf{1}_\mathcal{U}(0) \mathfrak{R}_K \log(H^n)
\right)
\end{align*}
We have $H^n\cdot  \mathbf{1}_\mathcal{U}\ast\mathbf{1}_\mathcal{U}(0) = H^n\cdot \vol(\mathcal{U}) = \vol(B_H^{\|\cdot\|})$, and so using Proposition \ref{prop:ExpectedPrimes} and the subsequent evaluation \eqref{exp_asymp} to keep track of $E_K(X;H)$ asymptotically, the above simplifies to
$$
\sim \left(1 - \frac{\log \vol(B_H^{\|\cdot\|})}{\log \vol(B_X^{\|\cdot\|})}
\right)E_K(X;H) \sim (1-\delta) E_K(X;H).
$$
This is exactly Conjecture \ref{conj:primevarianceNumberFields}.

That this computation simplifies so nicely and in such generality certainly suggests there is something more to be understood.

As for Theorem \ref{thm:SingularSeriesGeneralSum}, it should be possible to write down a version of Conjecture \ref{conj:primevarianceNumberFields} containing lower order terms of size $O(E_K(X;H)/\log X)$. We do not make a conjecture here for these lower order terms.

\section{Numerical evidence}

We have also found numerical support for Conjecture \ref{conj:primevarianceNumberFields} in the case of quadratic extensions, both real and imaginary. For a field $\mathbb{Q}(\sqrt{D})$ with ring of integers $\mathbb{Z}[\sqrt D]$, we considered short intervals induced by the norm $||a + b\sqrt{D}|| = \mathrm{max}(|a|,|b|)$, which we think of as small squares. For the fields $\mathbb{Q}(\sqrt D)$ with ring of integers $\mathbb{Z}\left[\frac{1 + \sqrt D}2\right]$, we instead considered the norm $\left|\left|a + b\left(\frac{1+\sqrt D}{2}\right)\right|\right| = \mathrm{max}(|a|,|b|)$, or equivalently $||a + b\sqrt D|| = \mathrm{max}(|a-b|,|2b|)$.

For each field $K = \mathbb{Q}(\sqrt D)$, for each large value $X$, we computed $V_K(X; X^\delta)$ for $\delta \in \frac 1{10}\mathbb{Z}$, $0 < \delta < 1$. We then plotted $\frac{V_K(X;X^\delta)}{E_K(X;X^\delta)}$ as a function of $\delta$, and observed that in all the quadratic field cases this plot approximated the line $1- \delta$. 

The main subtlety in our algorithm was in computing values of $\pi_K(x;H)$. To do this, for each $x + y \sqrt D$ with norm at most $X$ and $x, y \ge 0$, we counted the number of primes in the set $\{a + b \sqrt D \mid 0 \le a \le x, 0 \le b \le y\}$. We then used these counts at each of the corners of a given small square to solve for the number of primes within it. Whether or not elements are primes was tested using standard algorithms from algebraic number theory.

Below are plots of $\frac{V_K(X;X^\delta)}{E_K(X;X^\delta)}$, where $X = 1000$, as a function of $\delta$ for various quadratic number fields, both real and imaginary. In each case, the line $1-\delta$ is plotted.

For many of these cases, the convergence is rather slow, supporting the idea that there are lower-order terms which remain to be investigated.

\begin{figure}[h]
\centering
\begin{minipage}{.5\textwidth}
\includegraphics[width=0.9\linewidth]{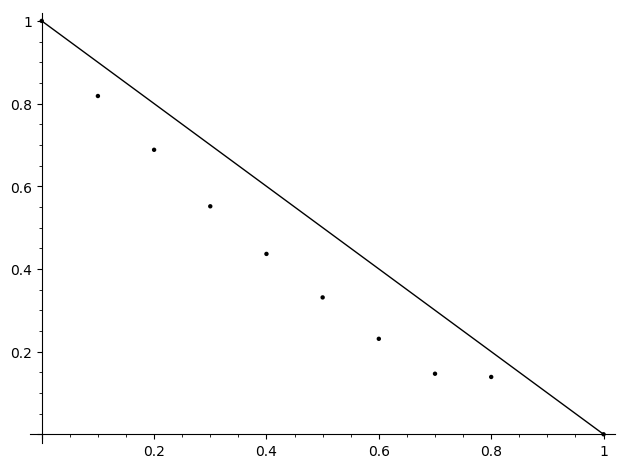}
\caption{$K = \mathbb{Q}(i)$}
\end{minipage}%
\begin{minipage}{.5\textwidth}
\includegraphics[width=0.9\textwidth]{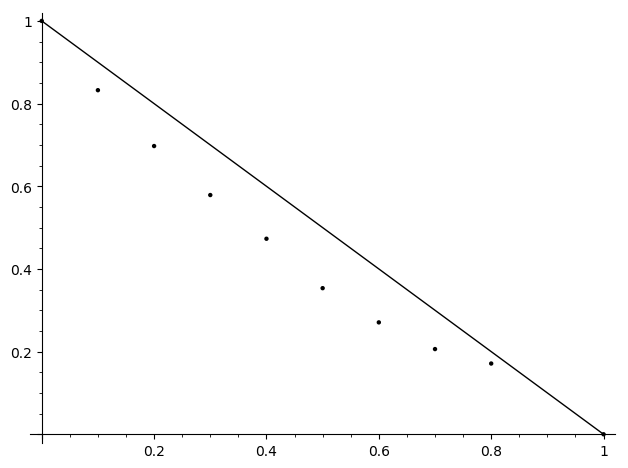}
\caption{$K = \mathbb{Q}(\sqrt{-3})$}
\end{minipage}
\end{figure}

\begin{figure}
\centering
\begin{minipage}{.5\textwidth}
\includegraphics[width=0.9\textwidth]{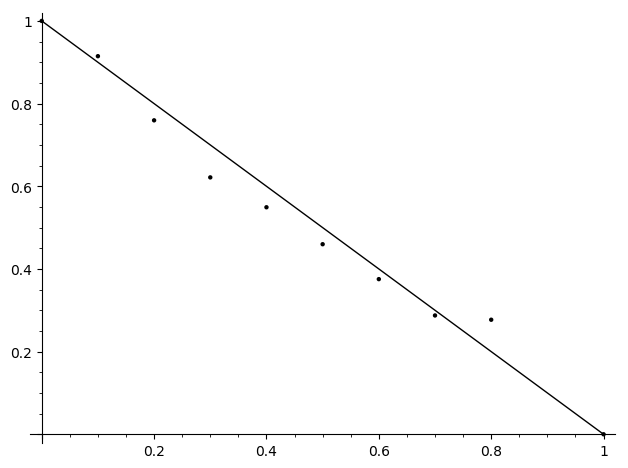}
\caption{$K = \mathbb{Q}(\sqrt{-5})$}
\end{minipage}%
\begin{minipage}{.5\textwidth}
\includegraphics[width=0.9\textwidth]{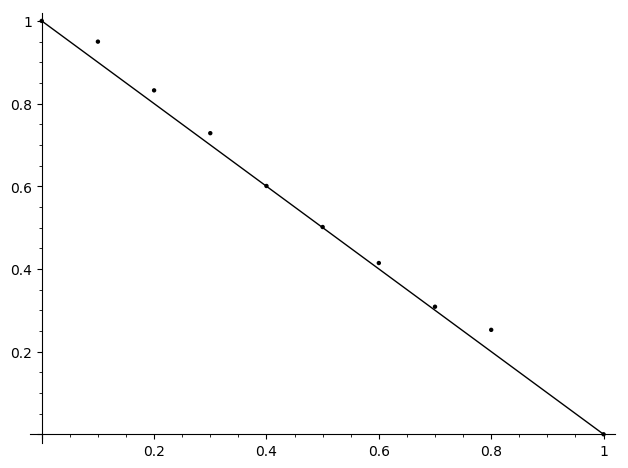}
\caption{$K = \mathbb{Q}(\sqrt{-7})$}
\end{minipage}
\end{figure}

\begin{figure}
\centering
\begin{minipage}{.5\textwidth}
\includegraphics[width=0.9\textwidth]{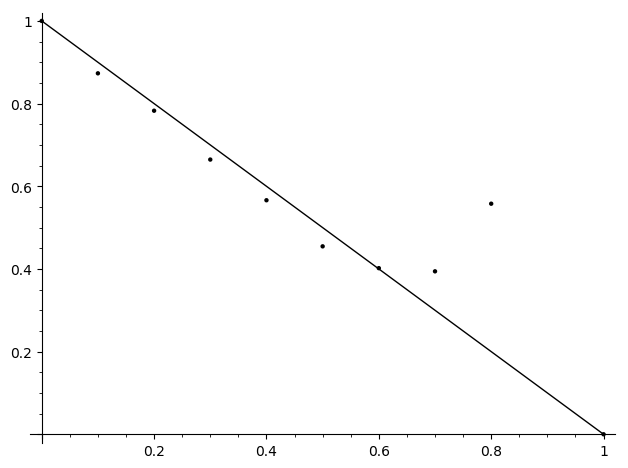}
\caption{$K = \mathbb{Q}(\sqrt{2})$}
\end{minipage}%
\begin{minipage}{.5\textwidth}
\includegraphics[width=0.9\textwidth]{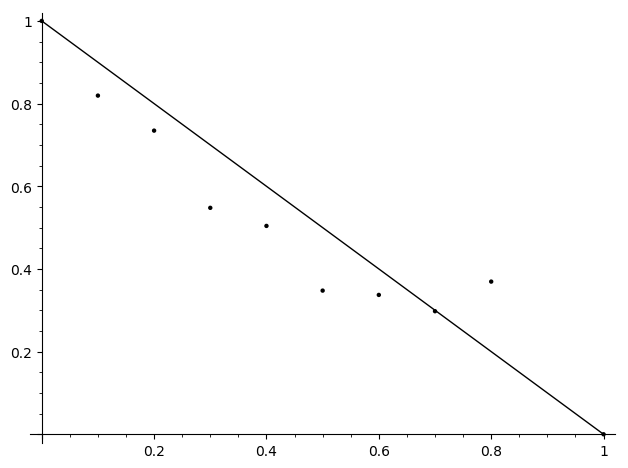}
\caption{$K = \mathbb{Q}(\sqrt{3})$}
\end{minipage}
\end{figure}

\begin{figure}
\centering
\begin{minipage}{.5\textwidth}
\includegraphics[width=0.9\textwidth]{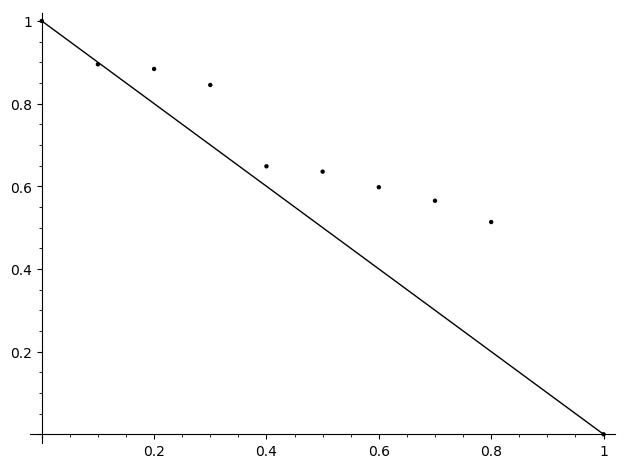}
\caption{$K=\mathbb{Q}(\sqrt{10})$}
\end{minipage}%
\end{figure}

\appendix
\section{Different ways to count primes in short intervals}
\label{TwoWeights}

We show here that there is no difference between \ref{conj:GoldstonMontgomery} and Conjecture \ref{conj:primevariance}:
\begin{prop}
\label{prop:TwoWeights}
Assume the Riemann hypothesis. Conjectures \ref{conj:GoldstonMontgomery} and \ref{conj:primevariance} are equivalent.
\end{prop}

In fact it may be seen that either of Conjectures \ref{conj:GoldstonMontgomery} or \ref{conj:primevariance} are stronger than the Riemann hypothesis; thus the assumption in this theorem is redundant. Nonetheless for convenience we make use of it in the proof below.

\begin{proof}
Define
$$
\widetilde{\psi}(x):= \sum_{n\leq x} \Lambda(n) - x,
$$
and recall that the Riemann hypothesis implies \cite{Ap}
$$
\widetilde{\psi}(x) = O(x^{1/2+\epsilon}).
$$
We also recall a result of Selberg \cite{Se}, that the Riemann hypothesis implies
\begin{equation}
\label{SelbergUpperBound}
\frac{1}{X} \int_0^X\Big| \sum_{x < n \leq x+H} \Lambda(n) -H \Big|^2 \, dx = O(H \log^2 X).
\end{equation}

Conjecture \ref{conj:GoldstonMontgomery} is equivalent to the claim
\begin{equation}
\label{conj1restate}
\frac{1}{X} \int_2^X \Big| \int_x^{x+H} d\widetilde{\psi}(t)\Big|^2 \, dx \sim H(\log X - \log H).
\end{equation}
(We now begin our integral at $2$ to avoid technicalities later in the proof.) We claim that Conjecture \ref{conj:primevariance} is equivalent to the claim
\begin{equation}
\label{conj2restate}
\frac{1}{X}\int_2^X \Big| \int_x^{x+H} \frac{d\widetilde{\psi}(t)}{\log t}  \Big|^2\, dx \sim \frac{H(\log X - \log H)}{\log^2 X}.
\end{equation}
Our plan of proof will thus be to show that \eqref{conj1restate} and \eqref{conj2restate} give the same information, but first we establish that \eqref{conj2restate} is equivalent to Conjecture \ref{conj:primevariance}:

The right hand side of \eqref{conj2restate} for $H = X^\delta$ is clearly
$$
\sim (1-\delta) E_\mathbb{N}(X;H),
$$
thus it remains only to show that the left hand side is approximated by $V_\mathbb{N}(X;H)$. But note that
\begin{align}
\label{pi_approx}
\notag \int_x^{x+H} \frac{d\widetilde{\psi}(t)}{\log t} &= \sum_{x < p \leq x+H} 1 + \sum_{\substack{x < p^k \leq x+H \\ k \geq 2}} \frac{1}{k} - \int_{x}^{x+H} \frac{dt}{\log t} \\
&= \widetilde{\pi}(x;H) + O(H X^{-1/2}),
\end{align}
as
$$
\sum_{\substack{x < p^k \leq x+H \\ k \geq 2}} \frac{1}{k} = \sum_{k\geq 2} \frac{1}{k}( \pi((x+H)^{1/k}) - \pi(x^{1/k}) \ll \sum_{k \geq 2} \frac{1}{k^2} H X^{1/k-1} \ll H X^{-1/2},
$$
using the trivial bound $\pi(b) - \pi(a) \ll b-a$ and $(x+H)^{1/k} - x^{1/k} \ll \tfrac{1}{k} H X^{1/k-1}$, and likewise as for $x\geq 2$,
$$
\int_x^{x+H} \frac{dt}{\log t} = \sum_{x < n \leq x+H} \frac{1}{\log n} + O\left(\frac{1}{\log x}
\right).
$$
Hence using \eqref{pi_approx} and the triangle inequality for $L^2$-norms,
$$
\left(\frac{1}{X} \int_2^X \Big| \int_x^{x+H} \frac{d\widetilde{\psi}(t)}{\log t}\Big|^2\, dx
\right)^{1/2} = (V_\mathbb{N}(X;H))^{1/2} + O(H X^{-1/2}).
$$
Because for $H = X^\delta$ we have $HX^{-1/2} = o(\sqrt{H/\log X})$, it is easily seen from this that Conjecture \ref{conj:primevariance} implies \eqref{conj2restate} and vice-versa.

It thus remains only to show that \eqref{conj1restate} and \eqref{conj2restate} are equivalent. Note that using integration by parts,
\begin{align*}
\int_x^{x+H} \frac{d\widetilde{\psi}(t)}{\log t} &= \frac{\widetilde{\psi}(x+H)}{\log(x+H)} - \frac{\widetilde{\psi}(x)}{\log x} + \int_x^{x+H} \frac{\widetilde{\psi}(t)}{t \log^2 t}\, dt \\
&= \frac{\widetilde{\psi}(x+H) - \widetilde{\psi}(x)}{\log x} + O(\frac{H}{X} X^{1/2+\epsilon}) + \int_x^{x+H} \frac{\widetilde{\psi}(t)}{t \log^2 t}\, dt \\
&= \frac{\widetilde{\psi}(x+H) - \widetilde{\psi}(x)}{\log x} + O(H X^{-1/2+\epsilon}).
\end{align*}
Hence again using the triangle inequality for $L^2$-norms,
\begin{multline}
\label{thetwoareclose}
\left(\frac{1}{X} \int_2^X \Big| \int_x^{x+H} \frac{d\widetilde{\psi}(t)}{\log t}\Big|^2\, dx
\right)^{1/2} \\= \left(\frac{1}{X} \int_2^X \Big| \frac{\widetilde{\psi}(x+H) - \widetilde{\psi}(x)}{\log x}\Big|^2\, dx
\right)^{1/2} + O(H X^{-1/2+\epsilon}).
\end{multline}
Because $\frac{1}{\log x} = \frac{1}{\log X}(1+O(\epsilon))$ for $x \geq X^{1-\epsilon}$, it is easy to deduce using Selberg's bound \eqref{SelbergUpperBound} that
\begin{multline*}
\frac{1}{X} \int_2^X \Big| \frac{\widetilde{\psi}(x+H) - \widetilde{\psi}(x)}{\log x}\Big|^2\, dx \\ = (1+o(1)) \frac{1}{X} \int_2^X \frac{|\widetilde{\psi}(x+H) - \widetilde{\psi}(x)|^2}{\log^2 X}\, dx + o(H/\log X).
\end{multline*}
Thus using \eqref{thetwoareclose} and again the fact that $H X^{-1/2+\epsilon} = o(\sqrt{H/\log X})$, it follows that \eqref{conj1restate} implies \eqref{conj2restate}, and vice-versa. This establishes the claimed equivalence.
\end{proof}

\Addresses


\begin{thebibliography}{99}


\bibitem{Ap} 
Apostol, T. M. \emph{Introduction to analytic number theory.} Undergraduate Texts in Mathematics. Springer-Verlag, New York-Heidelberg, 1976.


\bibitem{Br} 
Brandolini, L. Fourier transform of characteristic functions and Lebesgue constants for multiple Fourier series. \emph{Colloq. Math.} 65 (1993), no. 1, 51-59. 


\bibitem{CaHaLeOlPoTh} 
Castillo, A., Hall, C., Lemke Oliver, R.J., Pollack, P., Thompson, L.
Bounded gaps between primes in number fields and function fields. 
\emph{Proc. Amer. Math. Soc.} 143 (2015), no. 7, 2841-2856. 


\bibitem{Ga}
Gallagher, P. X. On the distribution of primes in short intervals. \emph{Mathematika} 23 (1976), no. 1, 4-9.


\bibitem{GaMu} 
Gallagher, P. X., Mueller, J. H.
Primes and zeros in short intervals. 
\emph{J. Reine Angew. Math.} 303/304 (1978), 205-220. 



\bibitem{GoMo} 
Goldston, D.A., Montgomery, H.L.  On pair correlations of zeros and
primes in short intervals. \emph{Analytic Number Theory and Diophantine Problems}
(Stillwater, OK, July 1984), Prog. Math. 70, Birkauser, Boston, 1987, 183-203.


\bibitem{GrSm} 
Gross, R., Smith, J. H. A generalization of a conjecture of Hardy and Littlewood to algebraic number fields. \emph{Rocky Mountain J. Math.} 30 (2000), no. 1, 195-215.



\bibitem{HaLi}
Hardy, G. H., Littlewood, J. E. Some problems of `Partitio numerorum'; III: On the expression of a number as a sum of primes. \emph{Acta Math.} 44 (1923), no. 1, 1-70.



\bibitem{IwKa} 
Iwaniec, H.; Kowalski, E. (2004). \emph{Analytic number theory.} Providence, R.I.: American Mathematical Society.


\bibitem{Ka}
Kaptan, D. A. A generalization of the Goldston-Pintz-Yildirim prime gaps result to number fields. \emph{Acta Math. Hungar.} 141 (2013), no. 1-2, 84-112. 


\bibitem{KeRu} 
Keating, J. P., Rudnick, Z. The variance of the number of prime polynomials in short intervals and in residue classes. \emph{Int. Math. Res. Not. IMRN} 2014, no. 1, 259-288. 


\bibitem{Mi} 
Mitsui, T. Generalized Prime Number Theorem, \emph{Jpn. J. Math.} 26 (1956), 1-42.


\bibitem{MoSo1}
Montgomery, H. L.; Soundararajan, K. Beyond pair correlation. \emph{Paul Erdős and his mathematics, I} (Budapest, 1999), 507-514, Bolyai Soc. Math. Stud., 11, János Bolyai Math. Soc., Budapest, 2002. 

\bibitem{MoSo2}
Montgomery, Hugh L.; Soundararajan, K. Primes in short intervals. \emph{Comm. Math. Phys.} 252 (2004), no. 1-3, 589-617.

\bibitem{Ra2} Rademacher, H. On the Phragm\'en-Lindel\"of theorem and some applications. \emph{Math. Zeit.} 72 (1959) no. 1, 192-204.

\bibitem{Ra} Rademacher, H. Zur additive Primzahltheorie algebraischer Zahlk\"orper, III. \emph{Math. Zeit.} 27 (1928) 319-426. 

\bibitem{Ri} Reiger, G.J. Ramanujansche Summen in algebraischen Zahlkorpern. \emph{Math. Nachr.} 22 (1960) 371-377.

\bibitem{Ri2} Rieger, G.J. Verallgemeinerung der Siebmethode von A. Selberg auf algebraische Zahlk\"orper. III \emph{J. reine angew. Math.} 208 (1961), 79-90.

\bibitem{Ro} 
Rodgers, B. (2013). Phd thesis: The statistics of the zeros of the Riemann zeta-function and related topics. UCLA.

\bibitem{Se} 
Selberg, A.
On the normal density of primes in small intervals, and the difference between consecutive primes. 
\emph{Arch. Math. Naturvid.} 47 (1943). no. 6, 87-105.


\bibitem{So}
Soundararajan, K.
The distribution of prime numbers. \emph{Equidistribution in number theory, an introduction} 59–83, 
NATO Sci. Ser. II Math. Phys. Chem., 237, Springer, Dordrecht, 2007. 


\bibitem{Sv} Svensson, I. Estimates for the Fourier transform of the characteristic function of a convex set. \emph{Arkiv f\"ur Matematik} 9 (1971) no. 1-2, 11-22.


\bibitem{TsZa}
Tsai, M.T., Zaharescu, A. On the distribution of algebraic primes in small regions. \emph{Manuscripta Math.} 145 (2014), no. 1-2, 111-123.






\end{thebibliography}
\end{document}